\documentclass[11pt]{amsart}
\usepackage{fullpage}
\usepackage{amsfonts}
\usepackage{amsmath}
\usepackage{amssymb}
\usepackage{amsthm}
\usepackage{color}
\usepackage{verbatim}
\newtheorem{theorem}{Theorem}
\newtheorem{example}{Example}

\newtheorem{cor}[theorem]{Corollary}
\newtheorem{question}{Question}

\newtheorem{lemma}[theorem]{Lemma}
\newtheorem{prop}[theorem]{Proposition}

\newcommand{\p}{P_{n,2d}}

\newcommand{\RR}{\mathbb{R}}
\newcommand{\CC}{\mathbb{C}}
\newcommand{\KK}{\mathbb{K}}
\newcommand{\PP}{\mathbb{P}}
\newcommand{\sq}{\Sigma_{n,2d}}
\newcommand{\ip}[2]{\langle #1,#2 \rangle }
\newcommand{\st}{\hspace{2mm} \mid \hspace{2mm}}
\newcommand{\binomi}[2]{\left(\genfrac{}{}{0 pt}{}{#1}{#2}\right)}
\newcommand{\rpn}{\mathbb{RP}^{n-1}}
\newcommand{\rank}{\mathrm{rank}}
\numberwithin{theorem}{section}
\numberwithin{equation}{section}
\numberwithin{question}{section}

\usepackage{amssymb,amsthm}
\usepackage[dvips]{graphicx}
\usepackage{float}

\newtheorem{itdefinition}[theorem]{Definition}

\newenvironment{definition}{\begin{itdefinition}\rm}{\end{itdefinition}}

\DeclareMathOperator{\satdeg}{satdeg}
\DeclareMathOperator{\reg}{reg}
\DeclareMathOperator{\Aff}{Aff}

\begin{document}
\author{Grigoriy Blekherman, Sadik Iliman, Martina Kubitzke}
\title{Dimensional Differences Between Faces of the Cones of Nonnegative Polynomials and Sums of Squares}

\address{Georgia Inst. of Technology, Atlanta, USA}
\email{greg@math.gatech.edu}
\address{Goethe-Universit\"at, FB 12 -- Institut f\"ur Mathematik,
Postfach 11 19 32, D-60054 Frankfurt am Main, Germany}
\email{\{iliman,kubitzke\}@math.uni-frankfurt.de}


\begin{abstract}
We study dimensions of the faces of the cone of nonnegative polynomials and the cone of sums of squares; we show that there are dimensional differences between corresponding faces of these cones. These dimensional gaps occur in all cases where there exist nonnegative polynomials that are not sums of squares. The gaps occur generically, they are not the product of selecting special faces of the cones. For ternary forms and quaternary quartics, we completely characterize when these differences are observed. Moreover, we provide an explicit description for these differences in the two smallest cases, in which
the cone of nonnegative polynomials and the cone of sums of squares are different. Our results follow from more general results concerning the relationship between the second ordinary power and the second symbolic power of the vanishing ideal of points in projective space.
\end{abstract}

\maketitle

\section{Introduction}


Let $H_{n,2d}$ denote the set of homogeneous polynomials (forms) in $n$ variables of degree $2d$ over $\mathbb R$ and let $\RR\PP^{n-1}$ resp. $\CC\PP^{n-1}$ denote the $(n - 1)$-dimensional real resp. complex projective space. 
For a fixed number of variables $n$ and degree $2d$, nonnegative polynomials and sums of squares form closed convex cones in $H_{n,2d}$. We call these cones $P_{n,2d}$ and $\Sigma_{n,2d}$, respectively, i.\,e.,

$$P_{n,2d}=\left\{p \in H_{n,2d}\hspace{2mm} | \hspace{2mm} p(x)\geq 0 \hspace{2mm} \text{for all} \hspace{2mm} x \in \RR\PP^{n-1}\right\},$$

$$\Sigma_{n,2d}=\left\{p \in P_{n,2d} \hspace{2mm} | \hspace{2mm} p(x)=\sum q_i^2  \hspace{2mm} \text{for some} \hspace{2mm} q_i \in H_{n,d} \right\}.$$

The relationship between the cone of nonnegative polynomials and the cone of sums of squares has been studied since Hilbert's seminal paper in 1888 \cite{Hilbert}. In this article, Hilbert showed that a nonnegative form in $n$ variables of even degree $2d$ has to be a sum of squares of forms only in the following cases: the form is bivariate, i.\,e., $n=2$, the form is quadratic, i.\,e., $2d=2$, or the form is a ternary quartic, i.\,e., $n=3$ and $2d=4$. In all other cases, he proved existence of nonnegative polynomials that are not sums of squares. It is remarkable that Hilbert's proof was existential and not constructive, and the first explicit nonnegative polynomial not being a sum of squares was found only 70 years later by Motzkin \cite{Rez1,Rez2}. Recently, in small dimensions, several aspects of the differences between these cones, such as the structure of the dual cones and the algebraic boundaries of these cones, were investigated (see \cite{Blekherman,Blekhermangorenstein,Blekhermanboundary}).

Understanding the precise relationship between the cones $P_{n,2d}$ and $\Sigma_{n,2d}$ is interesting from the point of view of computational complexity in polynomial optimization and also for practical testing for nonnegativity (see, e.\,g., \cite{Laurent}). Indeed, while testing whether a polynomial is nonnegative is NP-hard already in degree 4 \cite{NP}, testing whether a polynomial is a sum of squares can be reduced to a semidefinite programming problem, which can be solved efficiently \cite{pablo}. 

Unfortunately, except for the cases of $n=2$, the univariate case for nonhomogeneous polynomials, the case $2d=2$ (see \cite[Sections II.11 and II.12]{Sasha1}), and, to some extent, the case of ternary quartics, neither the structure of these cones nor their precise relationship with each other is very well understood. 

We focus on the study of faces of the cones $P_{n,2d}$ and $\Sigma_{n,2d}$, in particular on the investigation of their possible dimensions. A face $F$ of a convex set $K$ is called \emph{exposed} if there exists a supporting hyperplane $H$ such that $F=H\cap K$.
It is easy to describe exposed faces of $P_{n,2d}$. Indeed, the boundary of the cone $P_{n,2d}$ consists of all the forms with at least one zero, whereas its interior consists of all strictly positive forms. In particular, a maximal proper face of $P_{n,2d}$ consists of all forms with exactly one prescribed zero \cite[Chapter 4]{Blekherman:Parrilo:Thomas}.

Let $\Gamma$ be a set of distinct points in $\rpn$. The forms in $P_{n,2d}$ vanishing at all points of $\Gamma$ form an exposed face of $\p$, which we call $P_{n,2d}(\Gamma)$:
$$\p(\Gamma)=\{p \in \p \st p(s)=0 \hspace{2mm} \text{for all} \hspace{2mm} s \in \Gamma\}.$$

\noindent Similarly, we let $\sq(\Gamma)$ be the exposed face of $\sq$ consisting of forms that vanish at all points of $\Gamma$:
$$\sq(\Gamma)=\{p \in \sq \st p(s)=0 \hspace{2mm} \text{for all} \hspace{2mm} s \in \Gamma\}.$$

Moreover, any exposed face of $\p$ has a description of the above form, and the set $\Gamma$ can be chosen to be finite \cite[Chapter 4]{Blekherman:Parrilo:Thomas}.
We note that, despite this simple description of exposed faces, the full facial structure of $P_{n,2d}$ ``should" be very difficult to fully describe since --~as already mentioned~-- the problem of testing for nonnegativity is known to be NP-hard. 
Furthermore, even for the exposed faces $\p(\Gamma)$ and $\sq(\Gamma)$, except for the simple cases of $n=2$ and $2d=2,$ the possible dimensions of these faces have not been investigated yet. We close this gap by deriving estimates for the dimensions of the faces $\p(\Gamma)$ and $\sq(\Gamma)$ and by establishing dimensional differences between those faces in many cases. It is worth remarking that also Hilbert's original proof \cite{Hilbert} of existence of nonnegative polynomials that are not sums of squares can be viewed as establishing a dimensional gap of this type. This dimensional point of view was first made explicit in \cite{Rez2}.

For a generic set $\Gamma$ we reduce the question of dimensions of $\p(\Gamma)$ and $\sq(\Gamma)$ to the question of dimensions of the degree $2d$ components of certain ideals  associated with $\Gamma$. 
To simplify notation, we set $\mathbb{R}[x]=\mathbb{R}[x_1,\ldots,x_n]$. 
For an ideal $I \subset \mathbb{R}[x]$, let $I^2$ denote the second ordinary power of $I$, and let $I^{(2)}$ denote the second symbolic power of $I$. Moreover, let $I_d$ denote the homogeneous degree $d$ part of $I$.

If $I(\Gamma)\subset \mathbb{R}[x]$ is the vanishing ideal of a finite set of distinct points $\Gamma \subset \RR\PP^{n-1}$, then the second symbolic power $I^{(2)}(\Gamma)$ of $I(\Gamma)$ is the ideal of all forms in $\mathbb{R}[x]$ vanishing at every point of $\Gamma$ to order at least $2$:
$$I^{(2)}(\Gamma)=\{p \in \mathbb{R}[x] \,\, |\,\, \nabla p(s)=0 \hspace{2mm} \text{for all} \hspace{2mm} s \in \Gamma\}.$$

Since every nonnegative form that is zero on $s \in \Gamma$ must vanish to order 2 on $s$, it follows that the face $P_{n,2d}(\Gamma)$ is contained in the degree $2d$ part of $I^{(2)}(\Gamma)$:
$$P_{n,2d}(\Gamma) \subset I^{(2)}_{2d}(\Gamma).$$ 

On the other hand, we know that the face $\sq(\Gamma)$ is contained in the $2d$ part of the ordinary square of $I(\Gamma)$, i.\,e.,
$$\sq(\Gamma) \subset I_{2d}^{2}(\Gamma).$$  It is easy to see that this inclusion is actually full-dimensional, since we can choose a basis of $I_{2d}^2(\Gamma)$ consisting of squares and any nonnegative linear combination of these squares lies in $\sq(\Gamma)$. More precisely, the following is true.

\begin{prop}\label{prop:sosFace}
Let $\Gamma \subset \RR\PP^{n-1}$ be a finite set. Then $\Sigma_{n,2d}(\Gamma)$ is a full-dimensional convex cone in the vector space of all forms of degree $2d$ in $I^2_{2d}(\Gamma)$:
$$\dim \sq(\Gamma)=\dim I^{2}_{2d}(\Gamma).$$
\end{prop}

As for sums of squares, we pose the question, under which assumptions $P_{n,2d}(\Gamma)$ is a full-dimensional subcone of $I^{(2)}_{2d}(\Gamma)$. In order to answer this question, the following crucial definition is required.

\begin{definition}\label{def:dIndependence}
Let $\Gamma \subset \RR\PP^{n-1}$ be a finite set of distinct points and $I=I(\Gamma) \subset \mathbb{R}[x]$ be the vanishing ideal of $\Gamma$. We call $\Gamma \subset \RR\PP^{n-1}$ \emph{$d$-independent} if $\Gamma$ satisfies the following two conditions:
\begin{align}
&\label{first cond}\text{The forms in $I_d$ share no common zeroes in $\CC\PP^{n-1}$ outside of $\Gamma$. In other words, the}\\
&\nonumber \text{conditions of vanishing at $\Gamma$ force no additional zeroes on forms of degree $d$ in $H_{n,d}$.}\\
&\label{second cond}\text{For any $s \in \Gamma$ the forms that vanish to order 2 on $s$ and vanish at the rest of $\Gamma$}\\
&\nonumber \text{ to order 1 form a vector space of codimension $|\Gamma|+n-1$ in $H_{n,d}$.}&
\end{align}
\end{definition}

The second condition in the above definition simply states that the constraints of vanishing at $\Gamma$ and additionally double vanishing at any point $s \in \Gamma$ are all linearly independent.

The next proposition provides a sufficient condition for full-dimensionality of a face $P_{n,2d}(\Gamma)$ in $I^{(2)}_{2d}(\Gamma)$. 

\begin{prop}\label{prop:posFace}
Let $\Gamma \subset \RR\PP^{n-1}$ be a $d$-independent set. Then $P_{n,2d}(\Gamma)$ is a full-dimensional convex cone in $I^{(2)}_{2d}(\Gamma)$:
$$\dim P_{n,2d}(\Gamma)=\dim I^{(2)}_{2d}(\Gamma).$$
\end{prop}

Though we have seen that for nonnegative forms the condition for full-dimensionality of $P_{n,2d}(\Gamma)$ in $I^{(2)}_{2d}(\Gamma)$ is not as simple as for full-dimensionality of $\Sigma_{n,2d}(\Gamma)$ in $I^{2}_{2d}(\Gamma)$, it follows from the next two results that for generic $\Gamma$ the dimensions of $P_{n,2d}(\Gamma)$ and $I^{(2)}_{2d}(\Gamma)$ agree.

\begin{cor}\label{cor:Zariski}
The set of $d$-independent configurations of $k$ points in $\RR\PP^{n-1}$ is a Zariski open subset of $(\RR\PP^{n-1})^k$.
\end{cor}

\begin{prop}\label{prop:dIndependent}
Let $\Gamma$ be a generic collection of points in $\RR\PP^{n-1}$ such that $|\Gamma|\leq \binom{n+d-1}{d}-n$. Then $\Gamma$ is $d$-independent.
\end{prop}

In view of Propositions \ref{prop:sosFace}, \ref{prop:posFace} and \ref{prop:dIndependent} our original question of finding a dimensional difference between the faces $\p(\Gamma)$ and $\sq(\Gamma)$ can be reduced to the following:

\begin{question}
Let $\Gamma \subset \RR\PP^{n-1}$ (or equivalently $\CC\PP^{n-1}$) be a generic set of points such that $|\Gamma|\leq \binom{n+d-1}{d}-n$, and let $I(\Gamma)$ be the vanishing ideal of $\Gamma$. For what values of $|\Gamma|$ do we have $$I^{(2)}_{2d}(\Gamma)=I_{2d}^{2}(\Gamma)?$$
\end{question}
Indeed, for $n=3$, we show that this is true for all $d$-independent sets up to a certain size.

\begin{theorem}\label{thm:ternary}
Let $\Gamma$ be a $d$-independent set of points in $\RR\PP^2$ such that $|\Gamma| \leq \binom{d+1}{2}$. Then $$\dim I^{(2)}_{2d}(\Gamma)=\dim I^2_{2d}(\Gamma).$$
Moreover, if $\binom{d+1}{2} + 1\leq |\Gamma| \leq \binom{d+1}{2} + (d-2)$, then $\dim I^{(2)}_{2d}(\Gamma) > \dim I^2_{2d}(\Gamma)$.
\end{theorem}

As already explained, the next corollary is an immediate consequence of Theorem \ref{thm:ternary}. 

\begin{cor}
\label{cor:ternary}
Let $\Gamma\subset \RR\PP^2$ be $d$-independent with $|\Gamma| \leq \binom{d+1}{2}$. Then $$\dim P_{3,2d}(\Gamma)=\dim \Sigma_{3,2d}(\Gamma).$$
Furthermore, for $\binom{d+1}{2} + 1\leq |\Gamma| \leq \binom{d+1}{2} + (d-2)$ we have
$$\dim P_{3,2d}(\Gamma) > \dim \Sigma_{3,2d}(\Gamma).$$
\end{cor}

Similarly, for $n=4$, we derive the minimal size of a $d$-independent set $\Gamma$ such that the dimensions of the ideals  $I_4^2(\Gamma)$ and $I_4^{(2)}(\Gamma)$ are distinct. We will prove that if $\Gamma \subset \RR\PP^3$ with $|\Gamma|\leq 6$ is in general linear position, then it is $d$-independent.

\begin{theorem}\label{thm:diff4}
Let $\Gamma\subset \RR\PP^3$ be a finite set in general linear position. Then the following hold:
\begin{itemize}
 \item[(i)] If $|\Gamma|=6$, then 
\begin{equation*}
             \dim I_4^2(\Gamma)=10<11=\dim I_4^{(2)}(\Gamma).
            \end{equation*}
 \item[(ii)] If $|\Gamma|\leq 5$, then
\begin{equation*}
             \dim I_4^2(\Gamma)=\dim I_4^{(2)}(\Gamma).
            \end{equation*}
\end{itemize}
\end{theorem}

As a direct consequence, we obtain the following corollary.

\begin{cor}\label{cor:diff4}
Let $\Gamma\subset \RR\PP^3$ be a finite set in general linear position. Then the following hold:
\begin{itemize}
 \item[(i)] If $|\Gamma|=6$, then 
\begin{equation*}
             \dim \Sigma_{4,4}(\Gamma)=10<11=\dim P_{4,4}(\Gamma).
            \end{equation*}
 \item[(ii)] If $|\Gamma|\leq 5$, then
\begin{equation*}
             \dim \Sigma_{4,4}(\Gamma)=\dim P_{4,4}(\Gamma).
            \end{equation*}
\end{itemize}
\end{cor}

The paper is structured as follows. Section \ref{dimsection} focuses on proving that $d$-independence is a sufficient condition for full-dimensionality of $P_{n,2d}(\Gamma)$ in $I^{(2)}_{2d}(\Gamma)$. Additionally, properties of $d$-independent sets are studied, in particular, for sets with bounded size, $d$-independence turns out to be a generic condition. In Section \ref{Sec:ternary}, ternary forms are considered. Among other results, we prove Theorem \ref{thm:ternary}, which completely characterizes the occurence of dimensional differences between exposed faces $P_{3,2d}(\Gamma)$ and $\Sigma_{3,2d}(\Gamma)$. In Section \ref{Sec:sixpoints}, we are interested in the case $(n,2d) = (4,4)$ and derive analogous statements as for ternary forms (Theorem \ref{thm:diff4}), which yield a complete classification also in this case. In the two smallest cases ($(n,2d)\in\{((3,6),(4,4)\})$, we show that the maximal dimensional difference between the exposed faces is exactly one. Section \ref{Sec:1dim} provides explicit characterizations of these one-dimensional differences, yielding nonnegative polynomials that are not sums of squares. Finally, in Section \ref{gapsection}, we compute some a priori bounds for the possible dimensional differences as well as naive bounds for the minimum cardinality of a $d$-independent set such that a dimensional difference occurs. Moreover, we state some open questions for further research.

\section{Dimension of faces of $\p$}\label{dimsection}
In this section, we study the inclusion $P_{n,2d}(\Gamma)\subset I_{2d}^{(2)}(\Gamma)$ in more detail. In particular, we are interested in the question of when this inclusion is full-dimensional. Moreover, the proofs of 
Propositions \ref{prop:posFace} and \ref{prop:dIndependent}, such as Corollary \ref{cor:Zariski} are provided.
 
\subsection{Sum of Squares Certificate.}
In this section, we provide the proof of Proposition \ref{prop:posFace}.\\
Let $\Gamma$ be a finite set of points in $\rpn$ and consider $I^{(2)}_{2d}(\Gamma)$, the vector space of forms of degree $2d$ vanishing at $\Gamma$ with multiplicity at least $2$. 
Every double zero forces $n$ linear conditions on forms vanishing at $\Gamma$. Since not all of these conditions are necessarily independent, the following inequality always holds: 
\begin{equation}\label{eq:AH}
\dim I^{(2)}_{2d}(\Gamma)\geq \dim H_{n,2d}-n|\Gamma|.
\end{equation}
However, generically ~-- with a small list of exceptions~-- the Alexander-Hirschowitz Theorem \cite{Miranda} tells us that equality holds in \eqref{eq:AH}. 

We establish full-dimensionality of $\p(\Gamma)$ in $I^{(2)}_{2d}(\Gamma)$ by finding a form $p \in \p(\Gamma)$ to which we can add a suitably small multiple of any double vanishing form such that it will remain nonnegative:
$$p+\epsilon q \in \p(\Gamma) \hspace{2mm} \text{for some sufficiently small} \hspace{2mm} \epsilon \hspace{2mm} \text{and any} \hspace{2mm} q \in I^{(2)}_{2d}(\Gamma).$$

The form $p$ can be viewed as a certificate of full-dimensionality of $\p(\Gamma)$ in $I^{(2)}_{2d}(\Gamma)$. The important point is that $p$ can be any form, in particular, we will focus on finding such $p$ that is a sum of squares. This approach follows that of \cite{Rez2} and, indeed, it can be traced to the original proof of Hilbert \cite{Hilbert}.

For a form $p$, let the Hessian $H_p$ of $p$ be the matrix of second derivatives of $p$:

$$H_p=(h_{ij}), \hspace{5mm} \text{where} \hspace{5mm} h_{ij}=\frac{\partial^2 p}{\partial x_i \partial x_j}.$$

We note that if a form $p$ vanishes at a point $s$, then, by homogeneity, $p$ needs to vanish at a line through $s$. Therefore, $s$ lies in the kernel of the Hessian of $p$ at $s$: $H_p(s)s=0$.

If a form $p$ is nonnegative, then its Hessian at any zero $s$ is positive semidefinite since 0 is a minimum for $p$. We call a nonnegative form $p$ \textit{round} at a zero $s \in \rpn$ if the Hessian of $p$ at $s$ is positive definite on the subspace $s^{\perp}$ of vectors perpendicular to $s$, i.\,e., 
$$p \hspace{2mm} \text{is round at a zero} \hspace{2mm} s \hspace{2mm} \Leftrightarrow \hspace{2mm} y^TH_p(s)y > 0 \hspace{2mm} \text{for all } y\in s^{\perp} \hspace{2mm}. $$

\noindent For a form $p$, we let $Z(p)$ denote the real projective variety of $p$. 

We will need the following ``extension lemma'', which follows from Lemma 3.1 of \cite{Rez2}. 
\begin{lemma}\label{ext}
Let $p \in P_{n,2d}$ be a nonnegative form with a finite zero set $Z(p)$ and suppose that $p$ is round at every point in $Z(p)$. Furthermore, let $q$ be a form such that $q$ vanishes to order 2 on $Z(p)$. 
Then, for a sufficiently small $\epsilon$, the form $p+\epsilon q$ is nonnegative.
\end{lemma}

From this, we infer the following immediate corollary, which will be crucial for the proof of Proposition \ref{prop:posFace}.

\begin{cor} \label{ext2}
Let $\Gamma$ be a finite set in $\RR\PP^{n-1}$. Suppose that there exists a nonnegative form $p$ in $\p(\Gamma)$ such that 
$Z(p)=\Gamma$ and $p$ is round at every point $s \in \Gamma$. Then the face $\p(\Gamma)$ is full-dimensional in the vector space $I^{(2)}_{2d}(\Gamma)$.
\end{cor}
\begin{proof}
Let $p \in \p(\Gamma)$ be as in the assumptions. Then, by Lemma \ref{ext}, for any $q \in I^{(2)}_{2d}(\Gamma)$ we have $p+\epsilon q \in \p(\Gamma)$ for sufficiently small $\epsilon$. Since $\p(\Gamma)$ is a convex set, it follows that it is full-dimensional in $I^{(2)}_{2d}(\Gamma)$.
\end{proof}

We can finally provide the proof of Proposition \ref{prop:posFace}.

\begin{proof}[Proof of Proposition \ref{prop:posFace}]
Let $q_1, \ldots, q_k$ be a basis of $I_{d}(\Gamma)$. We claim that $p=\sum_{i=1}^{k} q_i^2$ has the properties of Corollary \ref{ext2} and therefore, the convex cone $\p(\Gamma)$ is full-dimensional in $I^{(2)}_{2d}(\Gamma)$.

Since $\Gamma$ forces no additional zeroes and since $q_1,\ldots,q_k$ is a basis of $I_{d}(\Gamma)$, it follows that the forms $q_i$ have no common zeroes outside of $\Gamma$ and thus $Z(p)=\Gamma$. 

Now choose $s \in \Gamma$. It remains to show that $p$ is round at $s$, i.\,e., $H_p(s)$ is positive definite on $s^{\perp}$. Since the forms in $I_{d}(\Gamma)$ that double vanish at $s$ form a vector space of codimension $n-1$ in $I_{d}(\Gamma)$, we see that for $1\leq i\leq k$ the gradients of $q_i$ at $s$ must span a vector space of dimension $n-1$. Since, by Euler's identity (see, e.\,g., \cite[Lemma 11.4]{Hassett}), $\ip{\nabla q_i}{s}=0$ for all $i$, this implies that the gradients actually span $s^{\perp}$.

Note that the Hessian of $p$ is the sum of the Hessians of $q_i^2$, i.\,e., 
$$H_p=\sum_{i=1}^k H_{q^2_i}.$$ Since $q_i(s)=0$ for all $i$ and $s \in \Gamma$, we conclude that

$$\frac{\partial^2 q_i^2}{\partial x_l \partial x_j}(s)=2\frac{\partial q_i}{\partial x_l}(s)\frac{\partial q_i}{\partial x_j}(s).$$

Therefore, we see that the Hessian of $q^2_i$ at any $s \in \Gamma$ is actually double the tensor of the gradient of $q_i$ at $s$ with itself: $$H_{q^2_i}(s)=2\nabla q_i\otimes \nabla q_i(s).$$ 
It is now straightforward to verify that 
\begin{equation*}
\nabla q_j(s)^T H_{q^2_i}(s)\nabla q_j(s) > 0
\end{equation*}
for all $1\leq i,j\leq k$ and $s\in\Gamma$, which shows the claim.
\end{proof}

From now on, we will focus on the study of the degree $2d$ part of the second symbolic power $I^{(2)}_{2d}(\Gamma)$ since, by Proposition \ref{prop:posFace}, we have the equality 
$$\dim P_{n,2d}(\Gamma)=\dim I^{(2)}_{2d}(\Gamma),$$
whenever $\Gamma$ is a finite $d$-independent set. We first provide the following equivalent characterization of $d$-independence.

\begin{prop}\label{prop:Hilbert}
Let $\Gamma\subset \RR\PP^{n-1}$ be a finite set of points and let $J\subset\mathbb R[x]$ be the ideal generated by $I_{d}(\Gamma)$. Then $\Gamma$ is $d$-independent if and only if the Hilbert polynomial of $\mathbb R[x]/J$ is equal to $|\Gamma|$. 
\end{prop}
\begin{proof}
Let $J$ be the ideal generated by $I_{d}(\Gamma)$. Let $I_{\CC,d}(\Gamma)$ be the set of all degree $d$ \emph{complex} forms vanishing at $\Gamma$. We first observe that linear combinations of forms in $I_{d}(\Gamma)$ taken with complex coefficients generate $I_{\CC,d}(\Gamma)$.
Therefore, if we let $J_{\CC}$ be the complex ideal generated $I_{\CC,d}(\Gamma)$, then it suffices to show that the Hilbert polynomial of $\mathbb C[x]/J_{\CC}$ is $|\Gamma|$.

We now apply Bertini's Theorem (see, e.g., \cite[Theorem 17.16]{AG}) 
to the linear system of divisors $I_{\CC,d}(\Gamma)$. Since $\Gamma$ is $d$-independent, it follows that a general element of $I_{\CC,d}(\Gamma)$ is non-singular at every point of $\Gamma$ and thus we may find a smooth hypersurface $f_1 \in I_{\CC,d}(\Gamma)$. Now $d$-independence guarantees that a general form in $I_{\CC,d}(\Gamma)$ intersects $f_1$ smoothly and therefore we apply Bertini's Theorem again to find $f_2 \in  I_{\CC,d}(\Gamma)$ such that $f_1\cap f_2$ is smooth. We proceed in this way repeatedly applying Bertini's Theorem until we end up with a transverse zero-dimensional intersection $V=f_1\cap\dots\cap f_{n-1}$. By construction we have $\Gamma \subseteq V$. 

Since $J_{\CC}$ defines a zero-dimensional ideal and all forms in $J_{\CC}$ vanish at $\Gamma$, it follows that the Hilbert polynomial of $\mathbb C[x]/J_{\CC}$ is a constant and it is greater or equal than $|\Gamma|$. Since we have $\langle f_1,\dots ,f_{n-1} \rangle \subset J_{\CC}$ and the ideal generated by $f_1,\dots,f_{n-1}$ is radical, it follows that we just need to show that for all $s \in V \setminus \Gamma$ there exists $g \in J_{\CC}$ such that $g(s) \neq 0$ and $g(z)=0$ for all $z\in V \setminus \{s\}$. Since $\Gamma$ is $d$-independent, it suffices to show for all $s \in V\setminus\Gamma$ there exists $h \in I_{d}(\Gamma)$ such that $h(s) \neq 0$. Also there exists $h' \in \CC[x_1,\dots,x_n]$ such that $h'(s)=1$ and $h'(z)=0$ for all $z\in V\setminus \{s\}$. Therefore $hh' \in J_{\CC}$ and one direction of the proposition follows.

Now suppose that $\Gamma$ is not $d$-independent. First, if all forms in $I_d(\Gamma)$ vanish at a point not in $\Gamma$, then all forms in $J$ vanish on at least $|\Gamma|+1$ points and therefore the Hilbert function of $\RR[x]/J$ is at least $|\Gamma|+1$ for all large enough degrees. Therefore the Hilbert polynomial of $\RR[x]/J$ is not $|\Gamma|$. Now suppose that for some $s \in \Gamma$ and some nonzero $w \in \RR^n$ we have $\langle \nabla p(s),w \rangle=0$ for all $p \in I_{d}(\Gamma)$. Then again we find that the Hilbert function of $\RR[x]/J$ is at least $|\Gamma|+1$ for all high enough degrees.

\end{proof}

We remark that in the above proof $h'$ may be chosen such that $\deg h' \leq (n-1)(d-1)$. Therefore we have $\deg hh' \leq (n-1)(d-1)+d$ and for all degrees $k \geq (n-1)(d-1)+d$ the Hilbert function of $\mathbb C[x]/J_{\CC}$ (and of $\mathbb R[x]/J$) evaluated at $k$ must be equal to $|\Gamma|$. Thus using standard methods (vanishing determinants) we can express the set of all configurations of $k$ points in $\mathbb{RP}^{n-1}$ 
that are $d$-independent as a complement of a closed algebraic set. Hence, the set of configurations $\Gamma$ of $k$ points in 
$\mathbb{RP}^{n-1}$ that are $d$-independent is Zariski open, which proves Corollary \ref{cor:Zariski}.


\subsection{A $d$-independent set of size $\binom{n+d-1}{d}-n$}
\label{Sec:dindependent}
The aim of this section is to provide the proof of Proposition \ref{prop:dIndependent}. For this goal, we will construct an example of a $d$-independent 
set of cardinality $\binom{n+d-1}{d}-n$. Since, by Corollary \ref{cor:Zariski}, being $d$-independent is a Zariski open condition, this will
already show the claim.

Define $\bar{S}_{n,d}$ to be the set of points in $\RR\PP^{n-1} $ that correspond to nonnegative integer partitions of $d$:

$$\bar{S}_{n,d}=\left\{[\alpha_1:\ldots:\alpha_n] \in \RR\PP^{n-1}\hspace{.4mm} \middle| \hspace{.9mm}\alpha_i \in \mathbb{Z}, \, \alpha_i \geq 0, \, \sum_{i=1}^n \alpha_i=d\right\}.$$

We can think of the points in $\bar{S}_{n,d}$ as all the possible exponent choices for monomials in $n$ variables of degree $d$. Therefore, $\bar{S}_{n,d}$ contains $\binom{n+d-1}{d}$ points.

Now let $S_{n,d}$ be the set of points in $\RR\PP^{n-1}$ that correspond to partitions of $d$ with at least 2 nonzero parts. The points in $S_{n,d}$ again correspond to monomials of degree $d$ but we exclude the monomials of the form $x_i^d$. Therefore, $S_{n,d}$ contains $\binom{n+d-1}{d}-n$ points.

\begin{prop}
 The set $S_{n,d}$ is $d$-independent.
\end{prop}

The proof of the above proposition requires some additional results. 
The following proposition is taken from \cite[p. 31]{Rez3} and has been known for at least a hundred years. We reproduce the proof below.

\begin{prop}\label{wtf}
There are no nontrivial forms in $H_{n,d}$ that vanish at $\bar{S}_{n,d}$. In other words, $I_d(\bar{S}_{n,d})=0$.
\end{prop}

\begin{proof}
For every point $s=[s_1:\ldots:s_n] \in \bar{S}_{n,d}$, we will construct a form $p_s \in \KK[x]_{n,d}$ that vanishes at all points in $\bar{S}_{n,d}$ except for $s$. This shows that the conditions of vanishing at any point in $\bar{S}_{n,d}$ are linearly independent and since $|\bar{S}_{n,d}|=\dim H_{n,d}$, we see that $\dim I_d(\bar{S}_{n,d})=0$.

Let $M=x_1+\ldots+x_n$. For $i=1, \ldots, n$, let $h_i$ be the form defined as follows:
$$h_i = \prod_{k=0}^{s_i-1}(dx_i-kM).$$
\noindent It is clear that the degree of $h_i$ is $s_i$ and $h_i$ vanishes at all partitions in $\bar{S}_{n,d}$ with $i$-th part less than $s_i$. Now let $p_s$ be defined as:

$$p_s=\prod_{i=1}^{n}h_i.$$

\noindent The form $p_s$ has degree $\sum_{i=1}^n s_i=d$, and it does not vanish at $s$. However, for any other partition of $d$, there exists $i$ such that the $i$-th part is less than $s_i$. Then, $h_i$ will vanish for that $i$ and, thus, $p_s$ will vanish at any partition of $d$ except for $s$.
\end{proof}

As in the proof of Proposition \ref{wtf} let $M=x_1+\ldots+x_n$. For $i=1, \ldots ,n$, define a form $Q_i$ as follows:

\begin{equation}\label{qi}
Q_i=\prod_{k=0}^{d-1} (dx_i-{kM}).
\end{equation}

 We observe that each $Q_i$ vanishes at $S_{n,d}$. Indeed, let $s=[s_1: \ldots: s_n]\in S_{n,d}$ and consider $Q_i(s)$. We know that $M(s)=d$, because points in $S_{n,d}$ are partitions of $d$, and, therefore, the factor of $Q_i$ that corresponds to $k=s_i$ will vanish at $s$, forcing $Q_i(s)=0$. Hence, $Q_i \in I_d(S_{n,d})$ for all $i=1, \ldots, n$.

We will now show that the forms $Q_i$ actually form a basis of $I_d(S_{n,d})$. The fact that we have such a nicely factoring basis is what, eventually, allows us to prove that $S_{n,d}$ is $d$-independent.

\begin{prop}\label{prop:basis}
The forms $Q_i$ form a basis of $I_d(S_{n,d})$.
\end{prop}

\begin{proof}
We first show that the forms $Q_i$ are linearly independent. Let $e_1, \ldots, e_n$ be the standard basis vectors of $\RR\PP^{n-1}$. It is easy to see that $Q_i(e_j)=0$ for $i \neq j$ since $x_i$ divides $Q_i$. On the other hand, $Q_i(e_i)=d!$. Therefore, if there exist $\alpha_i \in \RR$ such that $\alpha_1Q_1+\ldots+\alpha_nQ_n=0$, then, by evaluating this linear combination at $e_i$, we see that $\alpha_i=0$ and this works for all $i$. Thus, the forms $Q_i$ are linearly independent.

We now show that the forms $Q_i$ span $I_d(S_{n,d})$. Let $p\in I_d(S_{n,d})$ and let $\beta_i=p(e_i)$. Consider the form $$\bar{p}=p-\sum_{i=1}^{n}\frac{\beta_i}{d!}Q_i.$$ It is clear that $\bar{p}$ vanishes at the standard basis vectors $e_i$. Therefore, $\bar{p}$ vanishes not only on $S_{n,d}$ but also on $\bar{S}_{n,d}$. By Proposition \ref{wtf}, it follows that $\bar{p}=0$ and, therefore, $p$ is in the span of $Q_i$.

\end{proof}

We now show that the set $S_{n,d}$ satisfies the two conditions of $d$-independence from Definition \ref{def:dIndependence}.

\begin{lemma}
The set $S_{n,d}$ forces no additional zeroes for forms of degree $d$.
\end{lemma}

\begin{proof}
Since by Proposition \ref{prop:basis} the forms $Q_i$ form a basis of $I_d(S_{n,d})$, the statement of the lemma is equivalent to showing that $S_{n,d}$ is projectively equal to $\cap_{i=1}^n Z(Q_i).$

Let $v=[v_1: \ldots: v_n] \in \cap_{i=1}^n Z(Q_i)$ be a nonzero point and first suppose that $v_1+\ldots+v_n=0$, i.\,e., $M(v)=0$. Therefore, by Equation \eqref{qi}, we see that $Q_i(v)=d^dv_i^d$. Since by assumption $Q_i(v)=0$ for all $i$, it follows that $v=0$, which is a contradiction.

Now suppose that $v_1+\ldots+v_n \neq 0$. By homogeneity we can assume that $v_1+\ldots+v_n=d$. In this case, from Equation \eqref{qi}, it follows that $Q_i(v)=d^dv_i(v_i-1)\cdots(v_i-d+1).$ Since $Q_i(v)=0$ for all $i$, we infer that each $v_i$ is a nonnegative integer between 0 and $d-1$ and $v_1+\ldots+v_n=d$. In other words, $v \in S_{n,d}$.
\end{proof}

We know that $|S_{n,d}|=\binom{n+d-1}{d}-n$. For the second condition of $d$-independence, we need to show that for any $s \in S_{n,d}$ the vector space of forms double vanishing at $s$ and vanishing at the rest of $S_{n,d}$ with multiplicity 1 has codimension $|S_{n,d}|+n-1=\binom{n+d-1}{d}-1$ in $H_{n,d}$. Since $\dim H_{n,d}=\binom{n+d-1}{d}$, we thus need to show that the vector space of forms double vanishing at any $s\in S_{n,d}$ and vanishing at the rest of $S_{n,d}$ with multiplicity 1 is one-dimensional. This will follow from the next lemma.

\begin{lemma}
For every point $s \in S_{n,d}$, there is a unique (up to a constant multiple) form in $I_d(S_{n,d})$ singular at $s$.
\end{lemma}

\begin{proof}
Let $s=[s_1:\ldots: s_n] \in S_{n,d}$ and let $p \in I_d(S_{n,d})$ be a form singular at $s$.

Since by Proposition \ref{prop:basis} the forms $Q_i$ form a basis of $I_d(S_{n,d})$, we may assume that $p=\alpha_1Q_1+\ldots+\alpha_nQ_n$ for certain $\alpha_i \in \RR$. Now let $A=(a_{ij})$ be the $n \times n$ matrix with entries $$a_{ij}=\frac{\partial Q_i}{\partial x_j}(s).$$

The statement of the lemma is equivalent to showing that $\rank\,A = n-1$. Recall from Equation \eqref{qi} the definition of $Q_i$:

$$Q_i=\prod_{k=0}^{d-1} dx_i-{kM}.$$

The form $Q_i$ vanishes at $s$, because the term $dx_i-{s_iM}$ corresponding to $k=s_i$ vanishes at $s$. Therefore, the only nonzero term in $\displaystyle \frac{\partial Q_i}{\partial x_j}$ evaluated at $s$ will come from differentiating out $dx_i-{s_iM}$. Now, for $1\leq i\leq n$, let

$$P_{s_i}=\frac{Q_i}{dx_i-{s_iM}}.$$
We observe that $P_{s_i}(s) \neq 0$ since we removed from $Q_i$ the only factor that vanishes at $s$.

Recall that $M=x_1+\ldots+x_n$ and, therefore, if we differentiate out $dx_i-{s_iM}$ from $Q_i$ with respect to $x_j$ and evaluate it at $s$, we see that

 \begin{displaymath}
   \frac{\partial Q_i}{\partial x_j}(s) = \left\{
     \begin{array}{lr}
       P_{s_i}(s)\left(d-{s_j}\right) &  \text{if} \hspace{2mm}i=j\\
       -P_{s_i}(s)s_j  & \text{if} \hspace{2mm}i \neq j.
     \end{array}
   \right.
\end{displaymath}

Since $P_{s_i}(s) \neq 0$, we can divide the $i$-th row of $A$ by ${P_{s_i}(s)}$ to obtain a matrix $B=(b_{ij})$, where

 \begin{displaymath}
   b_{ij} = \left\{
     \begin{array}{lr}
       d-s_j &  \text{if} \hspace{2mm}i=j\\
       -s_j  & \text{if} \hspace{2mm}i \neq j.
     \end{array}
   \right.
\end{displaymath}

By construction, $\rank\, B = \rank\, A$.

Since $s$ is a partition of $d$, it is clear that the vector consisting of all 1s is in the kernel of $B$. Now let $C=(c_{ij})$ be the matrix with $j$-th column having the same entry $s_j$, i.\,e., $c_{ij}=s_j$. We observe that the rank of $C$ is 1 and $B=dI-C$, where $I$ is the identity matrix. Therefore, we know that $\text{rank} \hspace{.5mm} B \geq \text{rank} \hspace{.5mm} I - \text{rank} \hspace{.5mm} C=n-1$. Since we already found a vector in the kernel of $B$, it follows that the rank of $B$ is $n-1$.
\end{proof}

We have now shown that the set $S_{n,d}$ is $d$-independent and together with Corollary \ref{cor:Zariski} this shows that $d$-independence is a 
generic condition for sets of $k$ points in $\RR\PP^{n-1}$ with $k\leq \binom{n+d-1}{d}-n$. In particular, this proves Proposition \ref{prop:dIndependent}.

\section{Complete characterization of ternary forms}
\label{Sec:ternary}
In this section, we study the case of ternary forms and provide a complete characterization for the occurence of dimensional differences between exposed faces of the cones $P_{3,2d}$ and $\Sigma_{3,2d}$. 
The desired characterization will require some preparatory lemmas. 

\begin{lemma}\label{lem1}
 Let $\Gamma \subset\RR\PP^{n-1}$ be a finite set and $I = I(\Gamma)\subset \mathbb R[x]$ be the vanishing ideal. 
Let $\overline{I^m}$ be the saturation of the ordinary $m$-th power of $I$. Then 
$$\overline{I^m} = I^{(m)}.$$
\end{lemma}

\begin{proof}
Consider the primary decomposition $I = \bigcap_{p \in S} I(p)$, where $I(p) = \{f \in \mathbb R[x] \ | \ f(p) = 0\}$. Then, by definition of symbolic powers, we have
 $I^{(m)} = \bigcap_{p \in S} I(p)^m$. Moreover, the primary decomposition of $I^m$ contains $I(p)^m$ for all $p \in S$ and an additional component $Q$ associated
to the maximal homogeneous ideal $\mathfrak m = (x_1,\dots,x_n)$. By definition $\overline{I^m} = \bigcup_{j \geq 0} I^m : \mathfrak m^j$. From the previous discussion, we obtain  
\begin{align*}
I^m : \mathfrak m^j & = \left(\bigcap_{p \in S} I(p)^m  \cap Q\right) :  \mathfrak m^j \\
& = \left(\bigcap_{p \in S} I(p)^m : \mathfrak m^j\right)  \cap Q :  \mathfrak m^j \\
& = (I^{(m)} : \mathfrak m^j)  \cap (Q :  \mathfrak m^j) 
\end{align*}

Since $Q$ is $\mathfrak m$-primary, for sufficiently large $j$, it follows that $Q :  \mathfrak m^j = \mathbb R[x]$. Moreover,
$I^{(m)} : \mathfrak m^j = I^{(m)}$, thus, $I^m : \mathfrak m^j = I^{(m)}$, for sufficiently large $j$.
\end{proof}

Note that, as a consequence of the above lemma, we clearly have that $(\overline{I^2})_{2d} = (I^{(2)})_{2d}.$

Given an ideal $I\subset \mathbb R[x]$, let $\alpha(I)$ be the minimum degree of a generator of $I$.
We make the following simple observation.




\begin{lemma}\label{lem3}
 Let $I\subset \mathbb R[x]$ be a homogeneous ideal and let $m=\alpha(I)$. Then $(I^2)_{2m}=(I_m)^2$.
\end{lemma}

\begin{proof}
 Since $I_m\subset I$, we always have the inclusion $(I_m)^2\subset (I^2)_{2m}$. On the other hand, if $f\in (I^2)_{2m}$, then we can write 
$f=\sum_{a\in J}g_a\cdot h_a$ for certain polynomials $g_a,h_a\in I$ and $J\subset \mathbb{N}^n$ finite. Moreover, 
since $2m=\deg(f)=\deg(g_a)+\deg(h_a)$ and $\alpha(I)=m$, we conclude $\deg(f_a)=\deg(g_a)=m$ for all $a\in J$. Hence, $g_a,h_a\in I_m$ for all $a\in J$.
\end{proof}

 Recall that for an ideal  $I\subset \mathbb R[x]$ the \emph{saturation degree} of $I$ is defined as 
$$\satdeg(I) = \min\{s \ | \ \overline{I}_s = I_s\},$$
see, e.\,g., \cite{Eisenbud}. Moreover, we use $\reg(I)$ to denote the (Castelnuovo-Mumford) regularity of $I$. 

\begin{lemma}\label{lem4}
 Let $\Gamma \subset\RR\PP^{n-1}$ be a finite set and let $I = I(\Gamma)\subset \mathbb R[x]$ be the vanishing ideal. If $\alpha(I) = \reg(I)$, then $\satdeg(I^2) \leq \alpha(I^2) = 2\alpha(I)$.
\end{lemma}
 
\begin{proof}
As we have seen in the proof of Lemma \ref{lem1}, it holds that $I^2 = I^{(2)} \cap Q$, where $Q$ is $\mathfrak m$-primary. More precisely,
one has $Q = \mathfrak m^{2\alpha(I)}$. So, for $j \in \mathbb N$, one has $(I^2)_j = (I^{(2)} \cap \mathfrak m^{2\alpha(I)})_j$. 
For $j \geq 2\alpha(I)$, the ideal $\mathfrak m^{2\alpha(I)}$ contains all monomials of degree $j$. Therefore, $(I^2)_j = I^{(2)}_j$ in this case and,
by Lemma \ref{lem1}, it follows that $(I^2)_j = (\overline{I^2})_j$, i.\,e., $\satdeg(I^2) \leq 2\alpha(I)$. The equality $\alpha(I^2) = 2\alpha(I)$ is true
in general.
\end{proof}

We will need the following definition of a set $\Gamma\subset \mathbb R^n$ being in general linear position, which naturally extends to $\RR\PP^{n-1}$. 
\begin{definition}
 Let $\Gamma \subset\mathbb R^n$ be a finite set. $\Gamma$ is in \emph{general linear position} if one of the following conditions holds
\begin{itemize}
 \item[(i)] $|\Gamma| \leq n$ and $\dim \Aff(\Gamma) = |\Gamma| - 1$ 
\item[(ii)] $|\Gamma| \geq n+1$ and no $n$ of the points in $\Gamma$ lie on a common hyperplane.
\end{itemize}
\end{definition}

Note that any set $\Gamma$ with $|\Gamma| \leq n$ that is in general linear position can be extended to a set in general linear position of greater cardinality. 

\begin{lemma}\label{lem5}
Let $\Gamma \subset \RR\PP^{2}$ be in general linear position with $|\Gamma| = \binom{d+1}{2}$. Then $\reg(I_d(\Gamma)) = d$.
\end{lemma}

\begin{proof}
Due to the minimal resolution conjecture for $\RR\PP^{2}$ (see, e.\,g., \cite{BallicoGeramita,Walter}), for $\Gamma\subset \RR\PP^{2}$ it holds that $\reg(I_d(\Gamma)) \in \{d, d+1\}$.
For $|\Gamma| = \binom{d+1}{2}$ it is shown in \cite[Section 3]{Lorenzini} that, indeed, $\reg(I_d(\Gamma)) = d$.
\end{proof}

We can now prove Theorem \ref{thm:ternary}.

\begin{proof}[Proof of Theorem \ref{thm:ternary}]
Let $\Gamma \subset \RR\PP^2$ be such that $|\Gamma| = \binom{d+1}{2} - k$ for
$0\leq k\leq \binom{d+1}{2} - 1$. We first prove that $\dim I^2_{2d}(\Gamma) = \dim I^{(2)}_{2d}(\Gamma)$. We show the claim by induction on $k$. First assume $k=0$, i.\,e., $|\Gamma|=\binom{d+1}{2}$. 
By Lemma \ref{lem5}, we have $\alpha(I_d(\Gamma)) = \reg(I_d(\Gamma)) = d$ and, from Lemmas \ref{lem3} and \ref{lem4}, we infer 
that $I^{(2)}_{2d}(\Gamma) = I^2_{2d}(\Gamma)$.

Now suppose that the claim holds for a fixed $k$ and hence
we have $\dim I^2_{2d}(\Gamma) = \dim I^{(2)}_{2d}(\Gamma) =  \binom{d+2}{2} + 3k$ for any $d$-independent set $\Gamma\subset \RR\PP^{2}$ with $|\Gamma|=\binom{d+1}{2} - k$.
For the induction step $k\mapsto k+1$, let $\Gamma\subset \RR\PP^{2}$ be $d$-independent with $|\Gamma| =  \binom{d+1}{2} - (k+1)$. Let $p\in\RR\PP^{2}$ such that
the set $T=\Gamma\cup\{p\}$ is $d$-independent and in general linear position. Note that $|T| =   \binom{d+1}{2} - k$. Since $\Gamma$ and $T$ are $d$-independent, we have
\begin{equation*}
\dim I_{d}(T) = d + 1 + k \mbox{ and }\dim I_{d}(\Gamma) = d + 2 + k.
\end{equation*}
Furthermore, we know that $\dim I^{(2)}_{2d}(T) = \binom{d+2}{2} + 3k$ by hypothesis. Now, let $Q_1\in I_{d}(T)$ with the additional property that 
one of its partial derivatives vanishes at $p$. Without loss of generality, assume $\frac{\partial Q_1}{\partial x_1}(p) = 0$. We can extend $Q_1$ to a basis 
$B = \{Q_1, Q_2, \dots,Q_{d+1+k}\}$ of $I_{d}(T)$. Note that there must be at least one basis element $Q_j$, $j\neq 1$, such that $\frac{\partial Q_j}{\partial x_1}(p) \neq 0$. Otherwise, we arrive at a contradiction in regard to 
$d$-independence of $T$. By assumption, there exist 
$\binom{d+2}{2} + 3k$ pairwise products of elements of $B$ forming a basis of $I^{2}_{2d}(T)$. Let $\widetilde B$ denote this basis of $I^{2}_{2d}(T)$. 
Furthermore, extend $B$ to a basis of $I_{d}(\Gamma)$ by adding a suitable form $Q \in H_{3,d}$. Observe that $Q(p) \neq 0$ since $\Gamma$ is $d$-independent. 
We claim that there exist two forms 
$Q_l, Q_m \in B$ such that the set 
$$L = \widetilde B \cup \{Q_lQ, Q_mQ,Q^2\}$$
forms a basis of $I^{2}_{2d}(\Gamma)$. 
Suppose that this set $L$ is linearly dependent for any choice of $Q_l, Q_m \in B$. Hence, we have
$$\sum_{Q_iQ_j\in \widetilde B}^{}\alpha_{ij}Q_iQ_j + \alpha_{l,m}^{(l)}Q_lQ + \alpha_{l,m}^{(m)}Q_mQ + \alpha Q^2 = 0$$
for a nontrivial set of coefficients $(\alpha_{ij},\alpha_{l,m}^{(l)}, \alpha_{l,m}^{(m)}, \alpha)$. To simplify notation we use $P_{l,m}$ to denote the above linear combination. Clearly, we have 
$P_{l,m}(p) = \alpha Q^2(p)$. Since $Q(p) \neq 0$, it follows that $\alpha = 0$. We remark that the forms in $\widetilde B$ vanish to order $2$ at $p$, whereas the forms
$Q_lQ$ and $Q_mQ$ vanish to order $1$ at $p$. Therefore, by taking partial derivatives, we get
$$0 = \frac{\partial P_{l,m}}{\partial x_i}(p) = \alpha_{l,m}^{(l)}\frac{\partial Q_lQ}{\partial x_i}(p) + \alpha_{l,m}^{(m)}\frac{\partial Q_mQ}{\partial x_i}(p)$$
for $1\leq i\leq 3$. 
Assume that there exists $i$ such that $\frac{\partial Q_lQ}{\partial x_i}(p) = 0$ and $\frac{\partial Q_mQ}{\partial x_i}(p) \neq 0$. This implies
$\alpha_{l,m}^{(m)} = 0$. On the other hand, since $Q_lQ$ only vanishes to order $1$ at $p$, there exists $j\neq i$ such that $\frac{\partial Q_lQ}{\partial x_j}(p) \neq 0$. 
This forces $\alpha_{l,m}^{(l)} = 0$. Since, we assumed that $L$ is linearly dependent for all pairs $(Q_l,Q_m) \in B$, the just conducted reasoning shows that the following relation holds:
$$\frac{\partial Q_lQ}{\partial x_i}(p) = 0 \Leftrightarrow \frac{\partial Q_mQ}{\partial x_i}(p) = 0 \quad \textrm{for all}\quad i, l, m$$
and, hence, 
\begin{equation}\label{cond}
\frac{\partial Q_l}{\partial x_i}(p) = 0 \Leftrightarrow \frac{\partial Q_m}{\partial x_i}(p) = 0 \quad \textrm{for all}\quad i, l, m.
\end{equation}
Recall that the basis $B$ was constructed in the way that 
$\frac{\partial Q_1}{\partial x_1}(p) = 0$. Furthermore, as already remarked, due to $d$-independence of $T$, there has to exist $j\neq 1$ such 
that $\frac{\partial Q_j}{\partial x_1}(p) \neq 0$. Setting $l = 1$, $m = j$ and $i = 1$, this yields a contradiction to \eqref{cond}. Hence, there
have to exist two forms $Q_l,Q_m \in B$ such that $L$ is a linearly independent set. We have $|L| = \binom{d+2}{2} + 3(k+1) = \dim I^{(2)}_{2d}(\Gamma)$.\\
It remains to consider the case $\binom{d+1}{2} + 1\leq |\Gamma| \leq \binom{d+1}{2} + (d-2)$. It is routine to check that in this case
\begin{equation*}
\dim I^{(2)}_{2d}(\Gamma) -  \dim I_{2d}^2(\Gamma) = \left(\binom{d+2}{2} - 3|\Gamma|\right) - \left(\binom{\binom{d+2}{2}-|\Gamma|+1}{2}\right) > 0.
\end{equation*}
\end{proof}

From Theorem \ref{thm:ternary} we can immediately infer Corollary \ref{cor:ternary}.
Observe that the maximal size of a $d$-independent set $\Gamma \subset \RR\PP^2$ is equal to $\binom{d+2}{2}-3 = \binom{d+1}{2}+(d-2)$. Hence, Theorem \ref{thm:ternary} covers all $d$-independent sets concerning the occurence of dimensional differences.



\section{The case $(n,2d)=(4,4)$}
\label{Sec:sixpoints}

We now fully describe the situation with respect to $\p(\Gamma)$ and $\sq(\Gamma)$ for sets $\Gamma$ in $\mathbb{R}^4$ with $|\Gamma|\leq 6$.
In the following, we will work with affine representatives in $\mathbb{R}^4$ rather than projective points in $\mathbb{RP}^3$. 
Let $\Gamma=\{s_1, \ldots, s_6\}$ be a set of six points in $\mathbb{R}^4$ in general linear position so that any $4$ of them span $\mathbb{R}^4$. We will show that $\Gamma$ is 2-independent. In particular, this implies that the conditions of vanishing at $s_i \in \Gamma$ are linearly independent and therefore $\dim I_{1,2}(\Gamma)=\binom{5}{2}-6=4$. It follows that the dimension of the vector space $I^{2}_4(\Gamma)$ spanned by squares from $I_d(\Gamma)$ is at most $\binom{5}{2}=10$. We will show that it is equal to $10$.

On the other hand, the Alexander-Hirschowitz Theorem tells us that, generically, the dimension of $I_{2,4}(\Gamma)$ is $\binom{7}{4}-6 \cdot 4=11$. It is not hard to show that for 6 points in $\mathbb{R}^4$ in general linear position this dimension count is actually correct. Therefore, we should have a gap of 1 dimension between $P_{4,4}(\Gamma)$ and $\Sigma_{4,4}(\Gamma)$. 

We now state the main result of this section, which is Theorem \ref{thm:diff4} from the Introduction.

\begin{theorem}\label{thm:Difference4}
Let $\Gamma\subset \RR\PP^3$ be a finite set in general linear position. Then the following hold:
\begin{itemize}
 \item[(i)] If $|\Gamma|=6$, then 
\begin{equation*}
             \dim I_4^2(\Gamma)=10<11=\dim I_4^{(2)}(\Gamma).
            \end{equation*}
 \item[(ii)] If $|\Gamma|\leq 5$, then
\begin{equation*}
             \dim I_4^2(\Gamma)=\dim I_4^{(2)}(\Gamma).
            \end{equation*}
\end{itemize}
\end{theorem}

The proof of the above theorem will require several preparatory results. 
We start with a special construction for (i). 

To every $3$ element subset $T=\{t_1,t_2,t_3\}$ of $\{1, \ldots, 6\}$, we can associate the hyperplane $L_T$ spanned by the vectors $s_{t_1}$, $s_{t_2}$ and $s_{t_3}$.

We want to construct a double covering of $s_1, \ldots, s_6$ by 4 hyperplanes of the form $L_T$ with some nice combinatorial properties. We select 4 triples $T_i$ such that any two of them intersect in exactly one element of $\{1,\ldots,6\}$ and each element is contained in precisely two triples. Here is an example of such a covering, which is not unique: $$T_1=\{1,2,3\}, \hspace{4mm}T_2=\{1,4,5\},\hspace{4mm} T_3=\{2,4,6\}, \hspace{4mm}T_4=\{3,5,6\}.$$ To every such covering, we can associate the complementary covering, where we replace the triple $T_i$ with its complement $\overline{T}_i$. So, in the given example, $\overline{T}_1=\{4,5,6\}$, $\overline{T}_2=\{2,3,6\}$, $\overline{T}_3=\{1,3,5\}$ and $\overline{T}_4=\{1,2,4\}$. We observe that the complementary covering also shares the property that any two triples intersect in exactly one point and that every point is contained in exactly two triples.

To each triple $T$, we associate the linear functional with kernel $L_T$. We can think of this functional as the inner product with the unit normal vector to $L_T$, which is unique up to a sign. The choice of sign will not make a difference to us. We let $u_i$ and $v_i$ be a unit normal vector to $L_{T_i}$ and $L_{\overline{T}_i}$, respectively.

The vectors $u_i$ and $v_i$ form a pair of bases of $\mathbb{R}^4$. The key is to work with the dual configurations. We define $u_i^*$ to be vectors such that
 \begin{displaymath}
   \ip{u_i^*}{u_j} = \left\{
     \begin{array}{lr}
       1 &  \text{if} \hspace{2mm}i=j\\
       0  & \text{if} \hspace{3mm}i \neq j.
     \end{array}
   \right.
\end{displaymath}
One way to think about $u_i^*$ is that if we form the matrix $U$ with rows $u_i$, then $u_i^*$ form the columns of $U^{-1}$. We define vectors $v_i^*$ in the same way for $v_i$.

We will show that the four forms
\begin{eqnarray*}
Q_1(x) &=&\ip{x}{u_1}\ip{x}{v_1}, \hspace{4mm} Q_2(x)=\ip{x}{u_2}\ip{x}{v_2},\\
Q_3(x)&=&\ip{x}{u_3}\ip{x}{v_3}, \hspace{4mm} Q_4(x)=\ip{x}{u_4}\ip{x}{v_4}
\end{eqnarray*}
form a basis of $I_{2}(\Gamma)$. This factoring basis will allow us to prove 2-independence of $\Gamma$, and pairwise products $Q_iQ_j$ with $1\leq i \leq j\leq 4$ will form a basis of $I^{2}_4(\Gamma)$.

%

%

The vectors $u_i$ and $v_i$ are not just two arbitrary sets of bases of $\mathbb{R}^4$. Since they come from a configuration of 6 points in general linear position, they carry some structure. The following simple lemma will be crucial to our proofs.

\begin{lemma}\label{inprod}
For all $1\leq i,j\leq 4$ the following hold:
\begin{equation*}
\ip{u_i}{v_j^*}\neq 0 \hspace{2mm} \text{and} \hspace{2mm} \ip{v_i}{u_j^*}\neq 0.
\end{equation*}
\end{lemma}
\begin{proof}
By symmetry, it suffices to prove only one of the two assertions. Again, by symmetry, it will be enough to show that $\ip{u_1^*}{v_1}\neq 0$ and $\ip{u_1^*}{v_2}\neq 0$.

 Suppose that $\ip{u_1^*}{v_1}=0$. Then, it follows that $v_1$ is in the span of $u_2,u_3,u_4$. Let $$v_1=\alpha_2u_2+\alpha_3u_3+\alpha_4u_4.$$ Now   consider the inner product $\ip{v_1}{s_4}$. Recall that $v_1$ came from the triple $\{4,5,6\}$, $u_2$ from $\{1,4,5\}$, $u_3$ from $\{2,4,6\}$ and $u_4$ from $\{3,5,6\}$. It follows that $$\ip{v_1}{s_4}=0=\alpha_4\ip{s_4}{u_4}.$$

Being the points $s_i$ in general linear position implies that $\ip{s_4}{u_4}\neq 0$ and, therefore, $\alpha_4=0$. By considering inner products of $v_1$ with $s_5$ and $s_6$, we can also show that $\alpha_2=\alpha_3=0$, which yields a contradiction.

Similarly, if $\ip{u_1^*}{v_2}=0$, then $v_2$ is in the span of $u_2,u_3,u_4$. Let $$v_2=\alpha_2u_2+\alpha_3u_3+\alpha_4u_4.$$ Recall that $v_2$ came from the triple $\{2,3,6\}$, $u_2$ from $\{1,4,5\}$, $u_3$ from $\{2,4,6\}$ and $u_4$ from $\{3,5,6\}$. By an analogous argument as before, we can establish that $\alpha_2=0$ by taking inner products with $s_6$. Then, we use the inner product with $s_2$ to show that $\alpha_4=0$ and we will arrive at a contradiction.
\end{proof}

\begin{lemma}
\label{Lem:pairwise4}
The forms $Q_i$, $1\leq i\leq 4$ form a basis of $I_{2}(\Gamma)$. Furthermore, the pairwise products $Q_iQ_j$ with $1\leq i \leq j\leq 4$ form a basis of $I^{2}_4(\Gamma)$ and  $\dim I^{2}_4(\Gamma)=10$.
\end{lemma}
\begin{proof}
It is not hard to show that $I_{2}(\Gamma)$ has dimension 4. To show this claim, it suffices to prove that the polynomials $Q_i$ are linearly independent. Consider the values of $Q_i$ at the points $u_i^*$.

From the definition of the dual points $u_i^*$ and Lemma \ref{inprod}, it follows that $Q_i(u_i^*)=\ip{u_i^*}{v_i} \neq 0$ and $Q_i(u_j^*)= 0$ if $i \neq j$. Therefore, if $P=\alpha_1Q_1+\alpha_2Q_2+\alpha_3Q_3+\alpha_4Q_4=0$ for $\alpha_i \in \mathbb R$, then, by considering $P(u_i^*)$, we can see that $\alpha_i = 0$ for each $i$ and, therefore, the $Q_i$ are linearly independent.

Now consider pairwise products $Q_iQ_j$ for $1\leq i \leq j\leq 4$. These forms clearly belong to $I^{2}_4(\Gamma)$, and we need to show their linear independence. Of all the pairwise products only $Q^2_i$ does not vanish at $u^*_i$. Therefore, the squares $Q_i^2$ are linearly independent from all other pairwise products and it remains to show linear independence of $Q_iQ_j$ for $1\leq i < j\leq 4$.

By Lemma \ref{inprod}, only the products $Q_iQ_j$ vanish at $u^*_i$ to order 1. If both indices are distinct from $i$, then the product vanishes to order 2. Suppose that these products are linearly dependent, i.\,e., there exists a linear combination such that 
$$\sum_{1\leq i<j\leq 4}^{}\alpha_{ij}Q_iQ_j = 0,$$ where not all $\alpha_{ij}$ are zero. Differentiating and subsequently evaluating this equation at $u_1^*$, we obtain

$$\sum_{1 < j\leq 4}^{}\alpha_{1j}\frac{\partial Q_1Q_j}{\partial x_k}(u_1^*) = 0,\quad 1\leq k\leq 4,$$ which is equivalent to 
$$\sum_{1 < j\leq 4}^{}\alpha_{1j}\frac{\partial Q_j}{\partial x_k}(u_1^*) = 0,\quad 1\leq k\leq 4.$$ This is a $4\times 3$ system of linear equations in the variables $\alpha_{12}, \alpha_{13}, \alpha_{14}$. Denote the corresponding coefficient matrix by $A$. Assume that this system has a nontrivial solution, meaning that all $3\times 3$ minors of $A$ must vanish.
Considering the cross product of the three vectors $u_2, u_3, u_4 \in \mathbb R^4$ (see, e.\,g., \cite{Massey}), we can see that the entries of the cross product are exactly equal to an alternating $3\times 3$ minor of $A$. Hence, we conclude that the cross product is the zero vector, implying that the three vectors $u_2, u_3, u_4$ are linearly dependent, which is a contradiction since $u_1, u_2, u_3, u_4$ form a basis of $\mathbb R^4$. Hence, we conclude that $\alpha_{12}=\alpha_{13}=\alpha_{14} = 0$. Analogously, by following this procedure with $u_2^*, u_3^*$ and $u_4^*$, we can infer that $\alpha_{23}=\alpha_{24}=\alpha_{34} = 0$ and, hence, all pairwise products are linearly independent.
Since pairwise products $Q_iQ_j$ span $I^{2}_4(\Gamma)$ and are linearly independent, it finally follows that $\dim I^{2}_4(\Gamma)=10$.
\end{proof}

We are now ready to show $2$-independence of $\Gamma$.

\begin{prop}\label{prop:2-Independence}
Let $\Gamma$ be a set of $6$ points in $\mathbb{R}^4$ in general linear position. Then $\Gamma$ is 2-independent.
\end{prop}

\begin{proof}
We first show that $\Gamma$ forces no additional zeroes on quadratic forms. Recall that $Q_i=\ip{x}{u_i}\ip{x}{v_i}$ and the forms $Q_i$ form a basis of $I_{2}(\Gamma)$. It suffices to show that the forms $Q_i$ have no common zeroes outside of $\Gamma$.

Let $Z_{\mathbb C}(Q_i)$ denote the complex zero set of the forms $Q_i$, $1\leq i\leq 4$ and let $z$ be a nonzero point in the intersection $\cap_{i=1}^4 Z_{\mathbb C}(Q_i)$. First assume that $z \in \cap_{i=1}^4 Z_{\mathbb R}(Q_i)$. It follows that, for each $i$, we either have $\ip{z}{u_i}=0$ or $\ip{z}{v_i}=0$. Since $u_i$ and $v_i$ form a basis of $\mathbb{R}^4$, the vector $z$ cannot be orthogonal to all four $u_i$ or $v_i$. If $\ip{z}{u_i}=0$ for three indices $i$, which we may assume, without loss of generality, to be 1, 2 and 3, then it follows that $z$ is a multiple of $u_4^*$. But then $\ip{z}{u_4} \neq 0$  and, from Lemma \ref{inprod}, we know that $\ip{z}{v_4} \neq 0$. Therefore, $Q_4(z) \neq 0$, which is a contradiction.

It thus must happen that $z$ is orthogonal to at most 	two of the vectors $u_i$ and to two of the vectors $v_i$. Again, without loss of generality, we may assume that $z$ is orthogonal to $u_1$, $u_2$, $v_3$ and $v_4$. Since $u_1$ comes from the triple $\{1,2,3\}$, $u_2$ comes from $\{1,4,5\}$, $v_3$ comes from $\{1,3,5\}$ and $v_4$ comes from $\{1,2,4\}$, it follows that $z$ lies in the intersection of the spans of $\{s_1, s_2, s_3\}$, $\{s_1, s_4, s_5\}$, $\{s_1,s_3,s_5\}$ and $\{s_1,s_2,s_4\}$. Since the points $s_i$ are in general linear position, we infer that $s_1$ spans this intersection. The other points $s_i$ arise in the same manner from choosing different pairs of $u_i$'s and $v_i$'s. If $z$ is complex, then the same arguments as before applied to the real and imaginary part of $z$ imply the claim.

For the second condition of $2$-independence, we need to show that for any $s_i \in S$ there exists a unique (up to a constant multiple) form in $I_{2}(\Gamma)$ that is singular at $s_i$. Again, by symmetry, we only need to prove this for $s_1$. By construction, $s_1$ is orthogonal to $u_1$, $u_2$, $v_3$ and $v_4$. Therefore, it follows that
$\nabla Q_1(s_1)=\ip{v_1}{s_1}u_1$, $\nabla Q_2(s_1)=\ip{v_2}{s_1}u_2$, $\nabla Q_3(s)=\ip{u_3}{s_1}v_3$ and $\nabla Q_4(s)=\ip{u_4}{s_1}v_4$. The coefficients of the vectors $u_1$, $u_2$, $v_3$ and $v_4$ are nonzero, and since the points $s_i$ are in general linear position, it follows that $u_1$,$u_2$, $v_3$ and $v_4$ span the vector space $s_1^{\perp}.$ Therefore, there is only one (up to a constant multiple) linear combination of gradients of $Q_i$ that vanishes at $s_1$.
\end{proof}

Note that from the above proofs, it follows that $\dim \Sigma_{4,4}(\Gamma)=10$. On the other hand, by the Alexander-Hirschowitz Theorem we know that $\dim P_{4,4}(\Gamma)=11$. We have thus shown part (i) of Theorem \ref{thm:Difference4}.
In the remaining part of this section, we will provide the proof of Theorem \ref{thm:Difference4} (ii).  

\begin{prop}\label{prop:genInd}
 Let $\Gamma \subset \mathbb R^4$ be in general linear position with $|\Gamma| \leq 6$. Then $\Gamma$ is $2$-independent. 
\end{prop}

\begin{proof}
For $|\Gamma| = 6$, the statement is already proven in Proposition \ref{prop:2-Independence}. Let now $|\Gamma| \leq 5$. It is easy to see that the first condition of $2$-independence is still
satisfied whenever points are in general linear position. Indeed, we already know that the forms
$Q_1,\dots, Q_4$ do not have any zeroes outside of $\{s_1,\dots, s_6\}$. We define $Q_5(x) = \langle x,u_1\rangle\langle x,u_2\rangle$, which extends $Q_1,\dots,Q_4$ to a basis of $I_2(\Gamma)$. Since the points are in general linear position, we have 
$Q_5(s_6) \neq 0$. Hence, we see that the first condition of $2$-independence is satisfied for 
$|\Gamma|=5$. If $\Gamma$ is of smaller cardinality, we can always extend $\Gamma$ to a set $\widetilde{\Gamma}$ of cardinality $5$ that is in general linear position. 
In the next step, one can extend a basis $Q_1,\ldots,Q_5$ for $I_{2}(\widetilde{\Gamma})$ to a basis 
of $I_{2}(\Gamma)$. Using that those new polynomials do not vanish at all of $\widetilde{\Gamma}$ and using that $\widetilde{\Gamma}$ satisfies 
the first condition of $2$-independence, one can infer that also $\Gamma$ satisfies this condition.\\
It remains to verify the second condition of $2$-independence. I.\,e., we need to show that the vector space of quadratic forms vanishing at $\Gamma$ and that are singular 
at exactly one point of $\Gamma$ is of dimension $7-|\Gamma|$. We proceed by induction on $|\Gamma|$. For $|\Gamma|=6$, the claim follows from Proposition 6.3. For $|\Gamma|<6$, we extend $\Gamma$ 
to a set $\widetilde{\Gamma}$ with $|\widetilde{\Gamma}|=6$ in general linear position. By induction, we know that the vector space of quadratic forms vanishing at $\widetilde{S}$ and that 
are singular at one point of $\Gamma$ is of dimension $1$. This already implies that the dimension of the vector space of quadratic forms vanishing atly on $\Gamma$ and 
that are singular at exactly one point of $\Gamma$ is of dimension at most $1+|\widetilde{\Gamma}\setminus \Gamma|$. Hence, it suffices to show that, when decreasing the number of points 
in $\Gamma$, we gain one new linearly independent relation per point. We demonstrate this explicitly only for $|\Gamma|=5$ since all the other cases follow the same line of arguments. 
Let $P=\sum_{i=1}^5\alpha_iQ_i$ be singular at $s_1$, where $\alpha_i \in \mathbb R$. Then
$\nabla P(s_1)=\sum_{i=1}^4 \alpha_i\nabla Q_i(s_1)$. As in the proof of Proposition \ref{prop:2-Independence} (up to multiples), there is only one possible solution of this equation. Moreover,
$Q_5$ is clearly double-vanishing at $s_1$ and linearly independent of the former linear combination, which shows that the dimension of the vector space of quadratic forms vanishing at 
$s_1,\ldots,s_5$ and being singular at $s_1$ is $2$. For $s_2,\ldots,s_5$, the situation is a bit different since $\nabla Q_5(s_i)\neq 0$ for $2\leq i\leq 5$. If $P$ is singular at
$s_2$, then 
\begin{equation*}
\nabla P(s_2)=(\alpha_1\langle s_2,v_1\rangle+\alpha_5\langle s_2,u_2\rangle) u_1+ \alpha_2\langle s_2,u_2\rangle v_2+\alpha_3\langle s_2,v_3\rangle u_3+ \alpha_4\langle s_2,u_4\rangle v_4.
\end{equation*}
The same arguments as before show that $\alpha_1\langle s_2,v_1\rangle+\alpha_5\langle s_2,u_2\rangle,\alpha_2,\alpha_3,\alpha_4$ are uniquely determined (up to multiples). 
This gives one additional degree of freedom for choosing $\alpha_1$ and $\alpha_5$. Hence, again, the required dimension is $2$. For the other cases, one proceeds in an analogous way.
\end{proof}

We remark that for $n = 4$ the above proposition is stronger than the consequence of Proposition \ref{prop:dIndependent} since it explicitly classifies which generic point configurations in $\mathbb R^4$ are $2$-independent.

Consider the following forms
$$Q_5 = \langle x,u_1\rangle\langle x,u_2\rangle, Q_6 = \langle x,u_1\rangle\langle x,u_3\rangle, Q_7 = \langle x,u_2\rangle\langle x,u_4\rangle, Q_8 = \langle x,u_1\rangle\langle x,u_4\rangle.$$

Finally, we can provide the proof of Theorem \ref{thm:Difference4} (ii).

\begin{proof}[Proof of Theorem \ref{thm:Difference4}(ii)]
First, consider the case $|\Gamma| = 5$. Then, $\dim I_2(\Gamma) = 5$ and the forms $Q_1,\dots,Q_5$ form a basis of $I_2(\Gamma)$ (to see this, evaluate at $u_1^*,\dots,u_4^*$ and $s_6$). We claim that all $\binom{5+1}{2} = 15$ pairwise products $Q_iQ_j$, $1\leq i,j\leq 5$ are linearly independent, which is the right dimension count, since $\dim I_{4}^{(2)}(\Gamma) = 35 - 20 = 15$. Again, by evaluating at the points $u_1^*,\dots,u_4^*$ and $s_6$, we see that the pairwise products $Q_i^2$, $1\leq i\leq 5$ are linearly independent from the pairwise products $Q_iQ_j$ with $1\leq i < j\leq 5$. Hence, it remains to prove linear independence of the latter ones. For this aim, we use similar techniques as before. Suppose that those forms are linearly dependent, i.\,e., there exists a linear combination
\begin{equation}\label{eq:linComb} 
\sum_{1\leq i<j\leq 5}^{}\alpha_{ij}Q_iQ_j = 0,
\end{equation}
 where not all $\alpha_{ij}$ are zero.
Consider the evaluation of \eqref{eq:linComb} at the point $u_4^*$. The only forms that vanish with order $1$ at $u_4^*$ are $Q_1Q_4, Q_2Q_4$ and $Q_3Q_4$. The remaining ones vanish to higher order.
Hence, differentiating $P$ and subsequently evaluating at $u_4^*$ yields the following $4\times 3$ system of linear equations (note that $Q_i(u_i^*) \neq 0$):
$$\alpha_{14}\frac{\partial Q_1}{\partial x_k}(u_4^*) + \alpha_{24}\frac{\partial Q_2}{\partial x_k}(u_4^*) + \alpha_{34}\frac{\partial Q_3}{\partial x_k}(u_4^*) = 0,\quad 1\leq k\leq 4.$$
By the same arguments as in the proof of Lemma \ref{Lem:pairwise4}, we can conclude that $\alpha_{14}=\alpha_{24}=\alpha_{34} = 0$. Now evaluate \eqref{eq:linComb} at $u_3^*$. This time, the forms $Q_1Q_3, Q_2Q_3$ are the only pairwise products vanishing with order $1$. This yields a $4\times 2$ system of linear equations, from which we infer $\alpha_{13} = \alpha_{23} = 0$ since, in this case, vanishing of all $2\times 2$ minors implies linear dependence of $u_1, u_2$ contradicting the fact that $u_1,\dots,u_4$ form a basis of $\mathbb R^4$. Analogously, evaluating \eqref{eq:linComb} at $u_1^*$ yields $\alpha_{12}=\alpha_{15} = 0$ by exactly the same arguments as before. So, we are left with the pairwise products $Q_2Q_5, Q_3Q_5,Q_4Q_5$. These forms are clearly linearly independent since the forms $Q_i$, $2\leq i\leq 4$ are linearly independent. Hence, the claim follows.

Now, assume $|\Gamma| < 5$. In this case, note that there is always an overcount in the pairwise products. For example, in the case $|\Gamma| = 4$, there are $\binom{6+1}{2} = 21$ pairwise products. Since $\dim I_{4}^{(2)}(\Gamma) = 35 - 16 = 19$, we need to prove that, out of these $21$ pairwise products, there exist $19$ pairwise products that are linearly independent. Since the proof uses exactly the same strategy as in the case $|\Gamma| = 5$ and does not contain new arguments, in the next table, we only provide a basis for $I_{4}^2(\Gamma)$ for $|\Gamma| < 5$. For $|\Gamma| = m$, the forms $Q_i,\, 1\leq i\leq 10 - m$ form a basis of $I_2(\Gamma)$. We set $L = \{Q_iQ_j | i\leq j\}$ and use $B(\Gamma) \subset L$ to denote a basis of $I_{4}^2(\Gamma)$.

$$
\begin{tabular}{c|c|c|c}\hline 
 $|\Gamma| $ & $\dim I_{4}^2(\Gamma)$ & $\dim I_{4}^{(2)}(\Gamma)$ & $B(\Gamma)$\\ \hline 
 $4$ & $19$ & $19$ & $L\setminus\{Q_2Q_6,Q_3Q_5\}$ \\ \hline 
 $3$ & $23$ & $23$ & $L\setminus\{Q_1Q_7,Q_2Q_6,Q_3Q_5,Q_4Q_5,Q_5Q_6\}$\\ \hline 
 $2$ & $27$ & $27$ & $L\setminus\{Q_1Q_7,Q_2Q_6,Q_2Q_8,Q_3Q_8,Q_4Q_6,Q_5Q_7,Q_6Q_7,Q_6Q_8,Q_7Q_8\}$ \\ \hline 
\end{tabular}
$$

Note that, for $|\Gamma| = 1$, we always have $\dim I_{4}^2(\Gamma) = \dim I_{4}^{(2)}(\Gamma)$, and, hence, the proof is finished.
\end{proof}

From Theorem \ref{thm:Difference4} we can immediately infer Corollary \ref{cor:diff4}.

\section{Explicit characterization of $1$-dimensional differences}
\label{Sec:1dim}
The aim of this section is to explicitly characterize a dimensional difference between $P_{n,2d}(\Gamma)$ and $\Sigma_{n,2d}(\Gamma)$ for $(n,2d)=(4,4)$ and $(n,2d)=(3,6)$. By Theorems  \ref{thm:diff4} and \ref{thm:ternary}, the first time those differences occur is exactly for $|\Gamma|=6$, respectively, $|\Gamma|=7$, and it follows by dimension counting that the dimensional difference in these cases is exactly one. Hence, constructing forms in $I^{(2)}_{2d}(\Gamma)\setminus I^2_{2d}(\Gamma)$ already yields a complete characterization of the occuring dimensional differences in the two smallest cases, in which  nonnegative forms that are not sums of squares exist.

\subsection{The case $(n,2d)=(4,4)$}
In order to describe the dimensional difference, we need to construct a form $R$ of degree $4$ such that $R\in I^{(2)}_{4}(\Gamma)\setminus I^2_{4}(\Gamma)$.  

\begin{prop}\label{prop:diff}
Let $\Gamma\subset \mathbb{R}^4$ be in general linear position with $|\Gamma|=6$. Set $R=\ip{x}{u_1}\ip{x}{u_2}\ip{x}{u_3}\ip{x}{u_4}$. Then $R\in I^{(2)}_{4}(\Gamma)\setminus I^2_{4}(\Gamma)$.  
\end{prop}

\begin{proof}
We know that products $Q_iQ_j$ with $1\leq i \leq j\leq 4$ form a basis of $I^{2}_4(\Gamma)$. We observe that $R(u_i^*)=0$ for all $i$, and the only form from the spanning set that does not vanish at $u_i^*$ is $Q_i^2$. Therefore, if we assume that $R$ is spanned by $Q_iQ_j$, then $R$ needs to be spanned by products $Q_iQ_j$ with $i\neq j$.

Now consider $R(v_k^*)$. By Lemma \ref{inprod}, we know that $R(v_k^*) \neq 0$. However, $Q_iQ_j(v_k^*)=0$ since $\ip{v_i^*}{v_k}=0$ for $i \neq k$. Therefore, we arrive at a contradiction.
\end{proof}

\begin{cor}
 Let $\Gamma \subset \mathbb R^4$ be in general linear position with $|\Gamma| = 6$.
 There exists $p \in P_{4,4}(\Gamma)\setminus \Sigma_{4,4}(\Gamma)$ with $\Gamma \subset Z(p)$. These forms can be constructed 
via $Q_1^2 + Q_2^2 + Q_3^2 +Q_4^2 +\varepsilon R$ for
sufficiently small $\varepsilon > 0$.
\end{cor}

We now provide an explicit example for a form as described in the above corollary. 

\begin{example}\label{sectionexplicitexample} 
 Let $s_1=(0,0,1,1)$, $s_2=(0,1,0,1)$, $s_3=(0,1,1,0)$, $s_4=(1,0,0,1)$, $s_5=(1,0,1,0)$ and $s_6=(1,1,0,0)$ and $\Gamma=\{s_1,\ldots,s_6\}$. 
The following polynomials form a basis of $I_2(\Gamma)$.
\begin{align*}
  Q_1(x)&=x_1(x_1-x_2-x_3-x_4), \quad Q_2(x)=x_2(x_2-x_1-x_3-x_4), \\
 Q_3(x)&=x_3(x_3-x_1-x_2-x_4), \quad Q_4(x)=x_4(x_4-x_1-x_2-x_3)
\end{align*}
The form $R=\ip{x}{u_1}\ip{x}{u_2}\ip{x}{u_3}\ip{x}{u_4}$ from Proposition \ref{prop:diff} becomes $$R=\ip{x}{e_1}\ip{x}{e_2}\ip{x}{e_3}\ip{x}{e_4}=x_1x_2x_3x_4.$$
One can verify that $Q_1^2+Q_2^2+Q_3^2+Q_4^2+R\in P_{4,4}(\Gamma)\setminus \Sigma_{4,4}(\Gamma)$. 
\end{example}

\subsection{The case $(n,2d)=(3,6)$}
We consider the case $(n,2d) = (3,6)$ and $\Gamma\subset \mathbb{R}^3$ with $|\Gamma|=7$. Compared to the case $(n,2d)=(4,4)$ from the previous section,
the situation becomes more involved. 
Let $\Gamma=\{s_1,\dots,s_7\}$. 
Let $u_1,u_2,u_3$ be the normal vectors to the hyperplanes passing through $s_1,s_2$, respectively, $s_3,s_4$, respectively, $s_5,s_6$.
Note that, generically, $u_1,u_2,u_3$ are a basis of $\mathbb{R}^3$. Let $u_1^*,u_2^*, u_3^*$ 
be the dual basis to $u_1,u_2,u_3$. Furthermore, we define 
 $K_1, K_2, K_3$ to be the conics passing through the points $s_i$ with $i\in \{3,4,5,6,7\}$, respectively $i\in\{1,2,5,6,7\}$, respectively 
$i\in\{1,2,3,4,7\}$. Generically, we can assume that $K_i(u_j^*)\neq 0$ for $1\leq i\neq j\leq 3$.

\begin{lemma}
 Let $Q_1(x) = \langle x,u_1\rangle K_1$, $Q_2(x) = \langle x,u_2\rangle K_2$ and $Q_3(x) = \langle x,u_1\rangle K_3$. \\
 Then
$\{Q_1(x),Q_2(x),Q_3(x)\}$ is a basis of $I_3(\Gamma)$.
\end{lemma}

\begin{proof}
Assume that $Q_1(x), Q_2(x)$ and $Q_3(x)$ are linearly dependent, i.\,e., there exist $(\alpha_1,\alpha_2,\alpha_3)\in \mathbb{R}^3\setminus\{(0,0,0)\}$ such 
that 
\begin{equation}\label{eq:Ind}
0=\alpha_1  Q_1(x)+\alpha_2Q_2(x)+\alpha_3Q_3(x).
\end{equation}
Evaluating \eqref{eq:Ind} at $u_i^*$ and using that $\langle u_i^*,u_j\rangle=0 $, for $i\neq j$, and $K_i(u_j^*)\neq 0$, for all $1\leq i,j\leq 3$, we infer
$(\alpha_1,\alpha_2,\alpha_3)=(0,0,0)$. 
 \end{proof}

We now construct an explicit form $R\in I^{(2)}_{6}(\Gamma)\setminus I^2_{6}(\Gamma)$.

\begin{prop}
Let $R = K\langle x,u_1 \rangle \langle x,u_2 \rangle \langle x,u_3 \rangle$, where $K$ is the unique cubic double vanishing at the point $s_7$ and vanishing at 
$s_1,\dots, s_6$ with multiplicity 1 such that $K(u_i^*)\neq 0$ for $1\leq i\leq 3$. Then $R \in I^{(2)}_{6}(\Gamma)\setminus I^2_6(\Gamma)$.
\end{prop}

\begin{proof}
By construction, it is clear that $R \in  I^{(2)}_{6}(\Gamma)$. We have to prove that $R$ together with all pairwise products $Q_iQ_j$, $1\leq i\leq j\leq 3$ form a linearly independent set of polynomials.  
Suppose these forms were linearly dependent, i.\,e., there exists a nontrivial linear combination $\alpha_RR + \sum_{1\leq i\leq j\leq 3}^{} \alpha_{ij}Q_iQ_j = 0.$ By evaluating this relation 
at $u_i^*$, we get $\alpha_{ii} = 0$ for $1\leq i\leq 3$. Hence, the squares are linearly independent from the pairwise products $Q_iQ_j$ with $1\leq i < j\leq 3$ and it remains to
prove linear independence of the forms $Q_iQ_j$ with $1\leq i < j\leq 3$. Suppose that these forms are linearly dependent. The forms $Q_1Q_2$ and $Q_1Q_3$ vanish to order $1$ at 
$u_1^*$, whereas the forms $Q_2Q_3$ and $R$ vanish to order $2$. Hence, by taking the partial derivatives and subsequently evaluating them at $u_1^*$, we get for $k \in\{2,3\}$
$$\alpha_{1k}\frac{\partial Q_1Q_k}{\partial x_j}(u_1^*) = \alpha_{1k}\left(\frac{\partial Q_1}{\partial x_j}\cdot Q_k + Q_1\cdot\frac{\partial Q_k}{\partial x_j}\right)(u_1^*) = 
\alpha_{1k}Q_1(u_1^*)\cdot\frac{\partial Q_k}{\partial x_j}(u_1^*) = 0\,\,, \quad 1\leq j\leq 3$$
and hereby the following $3\times 2$ system of linear equations:
$$\alpha_{12}Q_1(u_1^*)\cdot\frac{\partial Q_2}{\partial x_j}(u_1^*) + \alpha_{13}Q_1(u_1^*)\cdot\frac{\partial Q_3}{\partial x_j}(u_1^*) = 0\,\,, \quad 1\leq j\leq 3.$$
Suppose that this system has a nontrivial solution. This is the case if and only if the rank of the corresponding coefficient matrix is one. Hence, all $2\times 2$ minors must vanish.
One can check (by taking the partial derivatives and considering $K_i(u_j^*) \neq 0$) that the rank is one if and only if the cross product of the two vectors 
$u_2$ and $u_3$ is zero, impliyng that $u_2$ and $u_3$ are linearly dependent. But this is a contradiction since $u_1, u_2$ and $u_3$ form a basis of $\mathbb R^3$. Hence, the above system can only have the 
trivial solution $\alpha_{12} = \alpha_{13} = 0$, and we are now left with the equation $\alpha_{23}Q_2Q_3 + \alpha_RR = 0$. However, since the form 
$Q_2Q_3$ vanishes to order $1$ at $u_2^*$ and the form $R$ with order $2$ at $u_2^*$, we get $\alpha_{23} = 0$ by taking the partial derivatives and, hence,  $\alpha_R = 0$, which
finishes the proof.
\end{proof}

\begin{cor}
 Let $\Gamma \subset \mathbb R^3$ be $3$-independent and $|\Gamma| = 7$. Under the assumptions of the previous lemma
 there exists $p \in P_{3,6}(\Gamma)\setminus \Sigma_{3,6}(\Gamma)$ with $\Gamma \subset Z(p)$. These forms can be constructed via $Q_1^2 + Q_2^2 + Q_3^2 + \varepsilon R$ for
sufficiently small $\varepsilon > 0$.
\end{cor}

\begin{example}
Note that the condition $K_i(u_j^*)\neq 0$ is essential.
Let $\Gamma=\{s_1,\ldots,s_7\}$ with $$s_1 = (1,0,0), s_2 = (0,1,0), s_3 = (0,0,1), s_4 = (1,1,0), s_5 = (1,0,1), s_6 = (0,1,1), s_7 = (1,1,1).$$ 
It can be checked that this set of points is $3$-independent. However, one can verify that $u_3^* = \left(\frac{1}{2}, \frac{1}{2}, 0\right)$ and hence it represents the same projective point as $s_4$. In particular, we always have $K_1(u_3^*) = 0$ by construction. However, we can perturb the point $s_4$ to $\tilde{s_4} = (1, -2, 2)$. The new set of points remains 
$3$-independent, and we have $K_i(u_j^*) \neq 0$ for $1\leq i\leq j\leq 3$. The three basis polynomials are then given by
\begin{align*}
Q_1(x_1,x_2,x_3) &= (3x_1x_2 - x_1x_3 - 2x_2x_3)(-x_3 + x_2 + x_1),\\
Q_2(x_1,x_2,x_3) &= (x_2 - x_3)(x_1 - x_3)(2x_1 + x_2),\\
Q_3(x_1,x_2,x_3) &= x_3(8x_1^2 + x_2^2 - 8x_1x_3 - x_2x_3).
\end{align*}
Furthermore, we have $$R =-x_3(2x_1 + x_2)^2(x_1 + x_2 - x_3)(x_2 - x_3)(-x_3 + x_1).$$
One can check that $$Q_1^2 + Q_2^2 + Q_3^2 + R \in P_{3,6}(\Gamma)\setminus \Sigma_{3,6}(\Gamma).$$
\end{example}

\section{General (naive) bounds for dimensional differences between $\p(\Gamma)$ and $\sq(\Gamma)$}
\label{gapsection}
We now want to derive some naive dimension counts for the dimensions of $\p(\Gamma)$ and $\sq(\Gamma)$, which help to understand when, theoretically, dimensional gaps between these faces can occur. Moreover, these counts yield some a priori bounds on the minimal size of $\Gamma$ such that dimensional gaps between these faces can be observed. A first step in this direction is the following lemma.

\begin{lemma}
\label{lem:naive}
Let $\Gamma$ be a $d$-independent set of $k$ points in $\RR\PP^{n-1}$. Then $\dim \p(\Gamma) \geq \binom{n+2d-1}{2d}-kn$ and $\dim\sq(\Gamma) \leq \binom{\binom{n+d-1}{d}-k+1}{2}$.
\end{lemma}

\begin{proof}
The dimension of $I^{(2)}_{2d}(\Gamma)$ is at least $\binom{n+2d-1}{2d}-kn$ since we are imposing at most $kn$ linearly independent conditions by forcing forms to double vanish at all points of $\Gamma$. From Proposition \ref{prop:posFace} we know that $\p(\Gamma)$ is full-dimensional in $I^{(2)}_{2d}(\Gamma)$ and, thus, the bound for the dimension of $\p(\Gamma)$ follows.

Since $\Gamma$ is $d$-independent, we know that the dimension of $I_d(\Gamma)$ is $\binom{n+d-1}{d}-k$. We can have at most $\binomi{\dim I_d(\Gamma)+1}{2}$ linearly independent pairwise products coming from $I_d(\Gamma)$ and, therefore, the dimension of $I^{2}_{2d}(\Gamma)$ is at most $\binomi{\binom{n+d-1}{d}-k+1}{2}$. Since the containment $\sq(\Gamma) \subset I^{2}_{2d}(\Gamma)$ is full-dimensional by Proposition \ref{prop:sosFace}, the bound for $\sq(\Gamma)$ follows.
\end{proof}

For a $d$-independent set $\Gamma$ of size $k$ let $G_{n,2d}(k)$ be the size of the minimal gap between the dimensions of $\p(\Gamma)$ and $\sq(\Gamma)$ which by Lemma \ref{lem:naive} is given by:
\begin{equation}
G_{n,2d}(k)=\binom{n+2d-1}{2d}-kn-\binomi{\binom{n+d-1}{d}-k+1}{2}.
\end{equation}

From Section \ref{Sec:dindependent} we know that there exist $d$-independent sets of any cardinality $k \leq \binom{n+d-1}{d}-n$. We want to determine the smallest positive integer $k$ for which $G_{n,2d}(k) > 0$, and we want to find the maximum of $G_{n,2d}(k)$.

\begin{prop}\label{uglybounds}
The function $G_{n,2d}(k)$ is maximized at $k=\binom{n+d-1}{d}-n$. Its value and the largest gap are
\begin{equation}\label{lessugly}
\binom{n+2d-1}{2d}-n\binom{n+d-1}{d}+\binom{n}{2}.
\end{equation}
The smallest value of $k$ such that $G_{n,2d}(k) > 0$ is the smallest integer strictly greater than:
\begin{equation}\label{fugliness}
\binom{n+d-1}{d}-n+\frac{1}{2}-\sqrt{\left(n-\frac{1}{2}\right)^2+2\binom{n+2d-1}{2d}-2n\binom{n+d-1}{d}}.
\end{equation}

\end{prop}

\begin{proof}
We observe that $G_{n,d}(k)$ is a quadratic function of $k$ with a negative leading coefficient. It is easy to show that $G_{n,d}(k)$ attains its maximum value at $\displaystyle k=\binom{n+d-1}{d}-n+\frac{1}{2}$. Therefore, the maximum value of $G_{n,2d}(k)$ for an integer $k$ will occur with $k=\binom{n+d-1}{d}-n$ and it is a matter of easy simplification to obtain equation \eqref{lessugly}.

The bound in equation \eqref{fugliness} comes from simply calculating the smallest root of $G_{n,2d}(k)$. We skip the routine application of the quadratic formula.
\end{proof}

We make several remarks. First we observe that the largest gap number from Proposition \ref{uglybounds}
$$\binom{n+2d-1}{2d}-n\binom{n+d-1}{d}+\binom{n}{2}$$
is zero in all cases, where the cones $\p$ and $\sq$ coincide. However, this number is strictly positive in the cases, where there are nonnegative forms that are not sums of squares. In the smallest cases $n=4$, $2d=4$ and $n=3$, $2d=6$, in which $\p$ is strictly larger than $\sq$, the gap number is 1.

However, as either $n$ or $d$ grows, we can see that the dimensional gap between exposed faces of $\p$ and $\sq$ grows and, asymptotically, it approaches the full dimension of the vector space $P_{n,2d}$.

We note that the bound from Equation \ref{fugliness} simplifies remarkably for $n=3$. In this case, we get the bound of $\binom{d+2}{2}-d-1$, and we need to take the smallest integer above that, which leads to
$$k=\binom{d+2}{2}-d = \binom{d+1}{2}+1.$$
This is actually the correct bound for the case of $n=3$ as we proved in Theorem \ref{thm:ternary}.

 Though, in general, for $n \geq 4$ the formula does not appear to simplify and the bound given is not going to be optimal, Theorem \ref{thm:Difference4} implies optimality of the naive dimension count also in the case $(n,2d) = (4,4)$. The non-optimality in the general case is caused by an overcount in the bound for the dimension of the vector space $I^2_{2d}(\Gamma)$.

We note that for $k=\binom{n+d-1}{d}-n$, which leads to the largest gap, the bound on the dimension of $I^2_{2d}(\Gamma)$ is also optimal, generically. We can see this from the example of the $d$-independent set $S_{n,d}$ from Section \ref{Sec:dindependent}, which has exactly this cardinality. Indeed, for  $S_{n,d}$, it is not hard to show that all pairwise products of the forms $Q_i$, which form the basis of $I_d(S_{n,d})$, are linearly independent in $H_{n,2d}$. This shows that the dimension of $I_{2d}^2(S_{n,d})$ is $\binom{n+1}{2}$, which is exactly equal to the bound we use.\\

\begin{question}From the above discussion it is natural to ask the following questions.
 \begin{itemize}
  \item[(i)] In addition to the known cases $((n,2d)\in\{(3,2d),(4,4)\})$, are there other cases, in which the naive bound from Proposition \ref{uglybounds} for the smallest cardinality of a set $\Gamma$ forcing a dimensional gap between the corresponding faces is correct?
\item[(ii)]  Given $n, d$, is it possible to characterize the set of integers $S$ such that $\dim \p(\Gamma) > \dim \sq(\Gamma)$ if and only if $|\Gamma| \in S$?
\item[(iii)]  Can the naive bounds from Proposition \ref{uglybounds} be further improved?
\item[(iv)]  Is it possible to characterize more cases for which $\dim \p(\Gamma) - \dim \sq(\Gamma) = 1$? 
 \end{itemize}

\end{question}

\section{Acknowledgments}
We would like to thank Aldo Conca for his explanations concerning symbolic powers of ideals.

\bibliography{biblio}
  \bibliographystyle{plain}

\end{document}